\documentclass[12pt,reqno]{amsart}
\usepackage{eurosym}
\usepackage{amsfonts}
\usepackage{amsmath}
\usepackage[bookmarksnumbered, colorlinks, plainpages]{hyperref}
\usepackage{graphicx}
\usepackage{epsfig}
\usepackage{epstopdf}
\usepackage{color}
\usepackage{caption}
\usepackage{subcaption}

\hypersetup{colorlinks=true,linkcolor=red, anchorcolor=green, citecolor=cyan, urlcolor=red, filecolor=magenta, pdftoolbar=true}
\newtheorem{theorem}{Theorem}[section]

\theoremstyle{definition}
\newtheorem{definition}[theorem]{Definition}
\newtheorem{example}[theorem]{Example}

\newtheorem{question}[theorem]{Question}

\theoremstyle{remark}

\numberwithin{equation}{section}

\begin{document}
\title[ Controlled rectangular metric type spaces and some applications to polynomial equations]
{Controlled rectangular metric type spaces and some applications to polynomial equations}

\author[Nabil Mlaiki]
{Nabil Mlaiki}  % in alphabetical order

\address{Nabil Mlaiki \newline
 Department of Mathematics and General Sciences, Prince Sultan University
 Riyadh, Saudi Arabia 11586}
\email{nmlaiki@psu.edu.sa, nmlaiki2012@gmail.com}

%\thanks{Submitted August .., .... Published .....}


\begin{abstract}
In this paper, we introduce a generalization of rectangular $b-$metric spaces, by changing the rectangular inequality as
follows
\begin{equation*}
\rho(x,y)\le \theta(x,y,u,v)[\rho(x,u)+\rho(u,v)+\rho(v,y)],
\end{equation*}%
for all distinct$\ x,y,u,v\in X.$ We prove some fixed-point theorems and
we use our results to present a nice application in last section of this paper.
Moreover, in the conclusion we present some new open questions.\newline
\end{abstract}

\maketitle

\section{Introduction}
In the last two decades, the generalization of metric spaces has been the focus of many researchers, and that is due to the importance of metric spaces and fixed point theory in solving open problems in many different fields. So, first of all we remind the reader of the definition of metric spaces.

\begin{definition}(Metric spaces)
Let $X$ be a nonempty set. A mapping $D:X^{2}\rightarrow [0,\infty)$ is called a metric on $X$ if for any $x,y,z \in X$ the following conditions are satisfied:\\
 ($R_1$) $x=y$ if and only if $D(x,y)=0$; \\
 ($R_2$) $D(x,y)=D(y,x)$;\\
 ($R_3$) $D(x,y)\leq D(x,z) +D(z,y)$.\\
In this case, the pair $(X,D)$ is called a metric space.
\end{definition}
A generalization of metric spaces to $b-$metric spaces was introduced, and we refer the reader to \cite{G2}. In 2017, Kamran in \cite{Kamran}, introduced an interesting generalization of the $b-$metric spaces called extended $b-$metric spaces and defined as follows.
\begin{definition}\cite{Kamran}
 Given a function $\theta: X\times X\rightarrow [1,\infty)$, where $X$ is a nonempty set.  The function $B:X\times X\rightarrow\mathbb [0,\infty)$ is called an extended $b$-metric if
\begin{enumerate}
\item $B(x,y)=0 \Longleftrightarrow x=y$;
\item $B(x,y) = B(y,x)$;
\item $B(x,y)  \leq \theta(x,y) [B(x,z) + B(z,y)]$,
\end{enumerate}
for all $x,y,z \in X$.
\end{definition}

In 2000, Branciari in \cite{Branc} introduced the concept of rectangular metric spaces. In 2015, George et al. in\cite{Geo}, generalized rectangular metric spaces to rectangular $b-$metric spaces. In this paper, and inspired by the work of Karman in \cite{Kamran}, we give a generalization to the rectangular $b-$metric spaces, called controlled rectangular $b-$metric spaces, but first we would like to remind the reader of the definitions of both spaces.
\begin{definition}\cite{Branc}\label{def3} (Rectangular (or Branciari) metric spaces))
Let $X$ be a nonempty set. A mapping $L:X^{2}\rightarrow [0,\infty)$ is called a rectangular metric on $X$ if for any $x,y \in X$ and all distinct points $u,v \in X\setminus \{x,y\}$, it satisfies the following conditions:\\
 ($R_1$) $x=y$ if and only if $L(x,y)=0$; \\
 ($R_2$) $L(x,y)=L(y,x)$;\\
 ($R_3$) $L(x,y)\leq L(x,u) + L(u,v)+L(v,y)$.\\
In this case, the pair $(X,L)$ is called a rectangular metric space.
\end{definition}

\begin{definition}\cite{Geo}\label{def5} (Rectangular $b-$metric spaces))
Let $X$ be a nonempty set. A mapping $L:X^{2}\rightarrow [0,\infty)$ is called a rectangular $b-$metric on $X$ if there exists a constant
$a\ge 1$ such that for any $x,y \in X$ and all distinct points $u,v \in X\setminus \{x,y\}$, it satisfies the following conditions:\\
 ($R_{b1}$) $x=y$ if and only if $L(x,y)=0$; \\
 ($R_{b2}$) $L(x,y)=L(y,x)$;\\
 ($R_{b3}$) $L(x,y)\leq a[L(x,u) + L(u,v)+L(v,y)]$.\\
In this case, the pair $(X,L)$ is called a rectangular $b-$metric space.
\end{definition}
As a generalization of the rectangular metric spaces and rectangular $b-$metric spaces, we define controlled rectangular $b-$metric spaces as follows.
\begin{definition}
Let $X$ be a non empty set, a function $\theta:X^{4}\rightarrow [1,\infty)$\\ and $\rho :X^{2}\rightarrow [0,\infty).$
We say that $(X,\rho)$ is a controlled rectangular $b-$metric space if all distinct $x,y,u,v \in X$ we have:
\begin{enumerate}
\item $\rho(x,y)=0$  if and only if $x=y;$
\item $\rho(x,y)=\rho(y,x);$
\item $\rho(x,y)\le \theta(x,y,u,v)[\rho(x,u)+\rho(u,v)+\rho(v,y)].$
\end{enumerate}
\end{definition}
\begin{definition}\label{d2}
Let $(X,\rho)$ be controlled rectangular $b-$metric space,
\begin{enumerate}
\item A sequence $\{x_{n}\}$ is called $\rho-$convergent in a controlled rectangular $b-$metric space $(X,\rho),$ if there exists $x\in X$
such that $\lim_{n\rightarrow \infty}\rho(x_{n},\nu)=\rho(\nu, \nu).$
\item A sequence $\{x_{n}\}$ is called $\rho-$Cauchy if and only if $\lim_{n,m\rightarrow \infty}\rho(x_{n},x_{m})
 \ \ \text{exists and finite}.$
\item A controlled rectangular $b-$metric space $(X,\rho)$ is called $\rho-$complete if for every $\rho-$Cauchy
sequence $\{x_n\}$ in $X$, there exists $\nu\in X$, such that $\lim_{n\rightarrow \infty}\rho(x_{n},\nu)= \lim_{n,m\rightarrow \infty}\rho_{r}(x_{n},x_m)=\rho_{r}(\nu,\nu).$
\item Let $a\in X$ define an open ball in a controlled rectangular $b-$metric space $(X,\rho)$ by $B_{\rho}(a, \eta)=\{b\in X\mid \rho(a,b)<\eta\}.$
\end{enumerate}
\end{definition}
Notice that, rectangular metric spaces and rectangular $b-$metric spaces are controlled rectangular $b-$metric spaces, but the converse is not always true. In he following example, we present a controlled rectangular $b-$metric space that is not a rectangular metric space.
\begin{example}
Let $X=Y\cup Z$ where $Y=\{\frac{1}{m}\mid m \ \ \text{is a natural number}\}$ and $Z$ be the set of positive integers.
We define $\rho :X^{2}\rightarrow [0,\infty)$ by
\[
\rho(x,y)=\begin{cases}
      0, &\text{if and only if} \ \ x=y\\
      2\beta, &\text{if}\ \ x,y\in Y\\
      \frac{\beta}{2}, &\textrm{otherwise},

   \end{cases}
   \]
where, $\beta$ is a constant bigger than $0.$
Now, define $\theta:X^{4}\rightarrow [1,\infty)$ by\\
$\theta(x,y,u,v)=\max\{x,y,u,v\}+2\beta.$
It is not difficult to check that $(X,\rho)$ is a controlled rectangular $b-$metric space.
However, $(X,\rho)$ is not a rectangular metric space, for instance notice that
$\rho(\frac{1}{2},\frac{1}{3})=2\beta> \rho(\frac{1}{2},2)+\rho(2,3)+\rho(3,\frac{1}{3})=\frac{3\beta}{2}.$

\end{example}

\section{Main Results}

\begin{theorem}\label{one}
Let $(X,\rho)$ is a controlled rectangular $b-$metric space, and $T$ a self mapping on $X.$
If there exists $0<k<1,$ such that  $$\rho(Tx,Ty)\le k\rho(x,y)$$ and
$$\sup_{m>1}\lim_{n\rightarrow \infty} \theta(x_{n},x_{n+1},x_{n+2},x_{m})\le \frac{1}{k},$$
then $T$ has a unique fixed point in $X.$
\end{theorem}
\begin{proof}
Let $x_{0}\in X$ and define the sequence $\{x_{n}\}$ as follows
$x_{1}=Tx_{0},x_{2}=T^{2}x_{0},\cdots,x_{n}=T^{n}x_{0},\cdots$
Now, by the hypothesis of theorem we have
\begin{align*}
\rho(x_{n},x_{n+1})&\le k\rho(x_{n-1},x_{n})\\
&\le  k^{2}\rho(x_{n-2},x_{n-1})\\ & \le \cdots\\ & \le  k^{n}\rho(x_{0},x_{1}).
\end{align*}
Note that, if we take the limit of the above inequality as $n\rightarrow \infty$
we deduce that
$\rho(x_{n},x_{n+1})\rightarrow 0\ \ \ \text{as} \ \ \ n\rightarrow \infty  \ \ \label{E1}$
Denote by $\rho_{i}=\rho(x_{n+i},x_{n+i+1})$
For all $n\geq 1,$ we have two cases.\\
\textbf{Case 1:} Let $x_n=x_m$ for some integers $n\neq m$. So, if for $m>n$ we have $T^{m-n}(x_n)=x_{n}$. Choose $y=x_{n}$ and $p=m-n$. Then
$T^py=y,$ and that is, $y$ is a periodic point of $T$. Thus,
$\rho(y,Ty)=\rho(T^py,T^{p+1}y)\leq  k^p \rho(y,Ty).$
Since $k\in (0,1)$, we get $\rho(y,Ty)=0$, so $y=Ty$, that is, $y$ is a fixed point of $T$.\\
\textbf{Case 2:} Suppose that $T^nx\neq T^mx$ for all integers $n\neq m$.
Let $n<m$ be two natural numbers, to show that $\{x_{n}\}$ is a $\rho-$Cauchy sequence, we need to consider two subcases:\\
Subcase 1: Assume that $m=n+2p+1.$ %%%%%%%%%%%
By property $(3)$ of the controlled rectangular $b-$metic spaces we have,
\vspace{-2cm}
\begin{align*}
\rho(x_{n},x_{n+2p+1})&\le \theta(x_{n},x_{n+1},x_{n+2},x_{n+2p+1})[\rho(x_{n},x_{n+1})+\rho(x_{n+1},x_{n+2})+\rho(x_{n+2},x_{n+2p+1})]\\
&\le \theta(x_{n},x_{n+1},x_{n+2},x_{n+2p+1})\rho(x_{n},x_{n+1})+ \theta(x_{n},x_{n+1},x_{n+2},x_{n+2p+1})\rho(x_{n+1},x_{n+2})\\ &+ \theta(x_{n},x_{n+1},x_{n+2},x_{n+2p+1})
\theta(x_{n+2},x_{n+3},x_{n+4},x_{n+2p+1})[\rho(x_{n+2},x_{n+3})\\ &+\rho(x_{n+3},x_{n+4})+\rho(x_{n+4},x_{n+2p+1})]\\&\le \theta(x_{n},x_{n+1},x_{n+2},x_{n+2p+1})\rho(x_{n},x_{n+1})+ \theta(x_{n},x_{n+1},x_{n+2},x_{n+2p+1})\rho(x_{n+1},x_{n+2})\\ &+ \theta(x_{n},x_{n+1},x_{n+2},x_{n+2p+1})
\theta(x_{n+2},x_{n+3},x_{n+4},x_{n+2p+1})\rho(x_{n+2},x_{n+3})\\ &+\theta(x_{n},x_{n+1},x_{n+2},x_{n+2p+1})
\theta(x_{n+2},x_{n+3},x_{n+4},x_{n+2p+1})\rho(x_{n+3},x_{n+4})\\
&+ \theta(x_{n},x_{n+1},x_{n+2},x_{n+2p+1})
\theta(x_{n+2},x_{n+3},x_{n+4},x_{n+2p+1})\rho(x_{n+4},x_{n+2p+1})
\\&\le \cdots\\
& \le \theta(x_{n},x_{n+1},x_{n+2},x_{n+2p+1})\rho(x_{n},x_{n+1})+ \theta(x_{n},x_{n+1},x_{n+2},x_{n+2p+1})\rho(x_{n+1},x_{n+2})\\ &+ \theta(x_{n},x_{n+1},x_{n+2},x_{n+2p+1})
\theta(x_{n+2},x_{n+3},x_{n+4},x_{n+2p+1})\rho(x_{n+2},x_{n+3})\\ &+\theta(x_{n},x_{n+1},x_{n+2},x_{n+2p+1})
\theta(x_{n+2},x_{n+3},x_{n+4},x_{n+2p+1})\rho(x_{n+3},x_{n+4})\\
&+ \cdots +\theta(x_{n},x_{n+1},x_{n+2},x_{n+2p+1})
\theta(x_{n+2},x_{n+3},x_{n+4},x_{n+2p+1})\\ &\cdots \theta(x_{n+2p-2},x_{n+2p-1},x_{n+2p},x_{n+2p+1}
)\rho(x_{n+2p},x_{n+2p+1})\\
& \le \theta(x_{n},x_{n+1},x_{n+2},x_{n+2p+1})\rho_{0}+ \theta(x_{n},x_{n+1},x_{n+2},x_{n+2p+1})\rho_{1}\\ &+ \theta(x_{n},x_{n+1},x_{n+2},x_{n+2p+1})
\theta(x_{n+2},x_{n+3},x_{n+4},x_{n+2p+1})\rho_{2}\\ &+\theta(x_{n},x_{n+1},x_{n+2},x_{n+2p+1})
\theta(x_{n+2},x_{n+3},x_{n+4},x_{n+2p+1})\rho_{3}\\
&+ \cdots \\ &+\theta(x_{n},x_{n+1},x_{n+2},x_{n+2p+1})
\theta(x_{n+2},x_{n+3},x_{n+4},x_{n+2p+1})\times \cdots \\
 &\times \cdots \theta(x_{n+2p-2},x_{n+2p-1},x_{n+2p},x_{n+2p+1}
)\rho_{2p}\\
& = \theta(x_{n},x_{n+1},x_{n+2},x_{n+2p+1})[\rho_{0}+ \rho_{1}]\\ &+ \theta(x_{n},x_{n+1},x_{n+2},x_{n+2p+1})
\theta(x_{n+2},x_{n+3},x_{n+4},x_{n+2p+1})[\rho_{2}+\rho_{3}]\\
&+ \cdots+\theta(x_{n},x_{n+1},x_{n+2},x_{n+2p+1})
\theta(x_{n+2},x_{n+3},x_{n+4},x_{n+2p+1})\times \cdots \\
 &\times \cdots \theta(x_{n+2p-2},x_{n+2p-1},x_{n+2p},x_{n+2p+1})[\rho_{2p-1}+\rho_{2p}]
\end{align*}
\begin{align*}
 &\le \theta(x_{n},x_{n+1},x_{n+2},x_{n+2p+1})[(k^{n}+k^{n+1})\rho(x_{0},x_{1})]\\ &+ \theta(x_{n},x_{n+1},x_{n+2},x_{n+2p+1})
\theta(x_{n+2},x_{n+3},x_{n+4},x_{n+2p+1})[(k^{n+2}+k^{n+3})\rho(x_{0},x_{1})]\\
&+ \cdots+\theta(x_{n},x_{n+1},x_{n+2},x_{n+2p+1})
\theta(x_{n+2},x_{n+3},x_{n+4},x_{n+2p+1})\times \cdots \\
 &\times \cdots \theta(x_{n+2p-2},x_{n+2p-1},x_{n+2p},x_{n+2p+1})[(k^{n+2p-2}+k^{n+2p-1})\rho(x_{0},x_{1})]\\ &
 \le [\theta(x_{n},x_{n+1},x_{n+2},x_{n+2p+1})(k^{n}+k^{n+1})\\ &+ \theta(x_{n},x_{n+1},x_{n+2},x_{n+2p+1})
\theta(x_{n+2},x_{n+3},x_{n+4},x_{n+2p+1})(k^{n+2}+k^{n+3})+ \\ &
\cdots+\theta(x_{n},x_{n+1},x_{n+2},x_{n+2p+1})
\theta(x_{n+2},x_{n+3},x_{n+4},x_{n+2p+1})\times \cdots\\ &\times \cdots \theta(x_{n+2p-2},x_{n+2p-1},x_{n+2p},x_{n+2p+1})(k^{n+2p-2}+k^{n+2p-1})]\rho(x_{0},x_{1})\\
& =\sum_{l=0}^{p-1}\prod_{i=0}^{l}\theta(x_{n+2i},x_{n+2i+1},x_{n+2i+2},x_{n+2p+1})[k^{n+2l}+k^{n+2l+1}]\rho(x_{0},x_{1})\\
&= \sum_{l=0}^{p-1}\prod_{i=0}^{l}\theta(x_{n+2i},x_{n+2i+1},x_{n+2i+2},x_{n+2p+1})[1+k]k^{n+2l}\rho(x_{0},x_{1})
\end{align*}
Now, using the fact that $k<1$ the above inequalities implies the following:
$$ \rho(x_{n},x_{n+2p+1})< \sum_{l=0}^{p-1}\prod_{i=0}^{l}\theta(x_{n+2i},x_{n+2i+1},x_{n+2i+2},x_{n+2p+1})2k^{n+2l}\rho(x_{0},x_{1}).$$
Since, $\sup_{m>1}\lim_{n\rightarrow \infty} \theta(x_{n},x_{n+1},x_{n+2},x_{m})\le \frac{1}{k},$ we deduce,
\begin{align*}
\lim_{n,p\rightarrow \infty}\rho(x_{n},x_{n+2p+1})&< \sum_{l=0}^{\infty}\prod_{i=0}^{l}\theta(x_{n+2i},x_{n+2i+1},x_{n+2i+2},x_{n+2p+1})2k^{n+2l}\rho(x_{0},x_{1})\\
&\le \sum_{l=0}^{\infty}\frac{1}{k^{l+1}}2k^{n+2l}\rho(x_{0},x_{1})\\
&\le \sum_{l=0}^{\infty}2k^{n+l-1}\rho(x_{0},x_{1}).
\end{align*}
Note that, the series $\sum_{l=0}^{\infty}2k^{n+l-1}\rho(x_{0},x_{1})$ converges by the ratio test, which implies that
$\rho(x_{n},x_{n+2p+1}) \ \ \text{converges} \ \ \text{as} \ \ n,p\rightarrow \infty.$\\
Subcase 2: $m=n+2p$
Fist of all, Note that
\begin{align*}
\rho(x_{n},x_{n+2})&\le k\rho(x_{n-1},x_{n+1})\\&\le k^{2}\rho(x_{n-2},x_{n})\\&\le \cdots \\&\le k^{n}\rho(x_{0},x_{2})
\end{align*}
which leads us to conclude that
$\rho(x_{n},x_{n+2}) \rightarrow 0 \ \ \text{as} \ \ n\rightarrow \infty.$
Similarly to Subcase 1 we have:\vspace{-3cm}
\begin{align*}
\rho(x_{n},x_{n+2p})&\le \theta(x_{n},x_{n+1},x_{n+2},x_{n+2p})[\rho(x_{n},x_{n+1})+\rho(x_{n+1},x_{n+2})+\rho(x_{n+2},x_{n+2p})]\\
&\le \theta(x_{n},x_{n+1},x_{n+2},x_{n+2p})\rho(x_{n},x_{n+1})+ \theta(x_{n},x_{n+1},x_{n+2},x_{n+2p})\rho(x_{n+1},x_{n+2})\\ &+ \theta(x_{n},x_{n+1},x_{n+2},x_{n+2p})
\theta(x_{n+2},x_{n+3},x_{n+4},x_{n+2p})[\rho(x_{n+2},x_{n+3})\\ &+\rho(x_{n+3},x_{n+4})+\rho(x_{n+4},x_{n+2p})]\\&\le \theta(x_{n},x_{n+1},x_{n+2},x_{n+2p})\rho(x_{n},x_{n+1})+ \theta(x_{n},x_{n+1},x_{n+2},x_{n+2p})\rho(x_{n+1},x_{n+2})\\ &+ \theta(x_{n},x_{n+1},x_{n+2},x_{n+2p})
\theta(x_{n+2},x_{n+3},x_{n+4},x_{n+2p})\rho(x_{n+2},x_{n+3})\\ &+\theta(x_{n},x_{n+1},x_{n+2},x_{n+2})
\theta(x_{n+2},x_{n+3},x_{n+4},x_{n+2p})\rho(x_{n+3},x_{n+4})\\
&+ \theta(x_{n},x_{n+1},x_{n+2},x_{n+2p})
\theta(x_{n+2},x_{n+3},x_{n+4},x_{n+2p})\rho(x_{n+4},x_{n+2p})\\ & \le
\theta(x_{n},x_{n+1},x_{n+2},x_{n+2p})\rho_{0}+ \theta(x_{n},x_{n+1},x_{n+2},x_{n+2p})\rho_{1}\\ &+ \theta(x_{n},x_{n+1},x_{n+2},x_{n+2p})
\theta(x_{n+2},x_{n+3},x_{n+4},x_{n+2p})\rho_{2}\\ &+\theta(x_{n},x_{n+1},x_{n+2},x_{n+2p})
\theta(x_{n+2},x_{n+3},x_{n+4},x_{n+2p})\rho_{3}\\
&+ \cdots \\ &+\theta(x_{n},x_{n+1},x_{n+2},x_{n+2p})
\theta(x_{n+2},x_{n+3},x_{n+4},x_{n+2p})\times \cdots \\
 &\times \cdots \theta(x_{n+2p-3},x_{n+2p-2},x_{n+2p-1},x_{n+2p}
)\rho_{2p}\\ &+ \prod_{i=0}^{2p-2} \theta(x_{n+2i},x_{n+2i+1},x_{n+2i+1},x_{n+2p})\rho(x_{n+2p-2
},x_{n+2p})\\
& =\sum_{l=0}^{p-1}\prod_{i=0}^{l}\theta(x_{n+2i},x_{n+2i+1},x_{n+2i+2},x_{n+2p+1})[k^{n+2l}+k^{n+2l+1}]\rho(x_{0},x_{1})\\
&+ \prod_{i=0}^{2p-2} \theta(x_{n+2i},x_{n+2i+1},x_{n+2i+1},x_{n+2p})\rho(x_{n+2p-2
},x_{n+2p})\\
&= \sum_{l=0}^{p-1}\prod_{i=0}^{l}\theta(x_{n+2i},x_{n+2i+1},x_{n+2i+2},x_{n+2p+1})[1+k]k^{n+2l}\rho(x_{0},x_{1})\\
&+\prod_{i=0}^{2p-2} \theta(x_{n+2i},x_{n+2i+1},x_{n+2i+1},x_{n+2p})\rho(x_{n+2p-2
},x_{n+2p})\\
&\le \sum_{l=0}^{p-1}\prod_{i=0}^{l}\theta(x_{n+2i},x_{n+2i+1},x_{n+2i+2},x_{n+2p+1})[1+k]k^{n+2l}\rho(x_{0},x_{1})\\
&+\prod_{i=0}^{2p-2} \theta(x_{n+2i},x_{n+2i+1},x_{n+2i+1},x_{n+2p})k^{n+2p-2}\rho(x_{0
},x_{2})
\end{align*}
Since, $\sup_{m>1}\lim_{n\rightarrow \infty} \theta(x_{n},x_{n+1},x_{n+2},x_{m})\le \frac{1}{k},$ we deduce,
\begin{align*}
\lim_{n,p\rightarrow \infty}\rho(x_{n},x_{n+2p})
&\le \lim_{n,p\rightarrow \infty}\sum_{l=0}^{p-1}\frac{1}{k^{l+1}}[1+k]k^{n+2l}\rho(x_{0},x_{1})
+k^{2p-1}k^{n+2p-2}\rho(x_{0
},x_{2})\\
&= \lim_{n,p\rightarrow \infty}\sum_{l=0}^{p-1}[1+k]k^{n+l-1}\rho(x_{0},x_{1})
+k^{n-1}\rho(x_{0
},x_{2})\\
&\le \sum_{m=0}^{\infty}[1+k]k^{m}\rho(x_{0},x_{1})
+k^{m}\rho(x_{0
},x_{2})
\end{align*}
By using the Ratio Test, it is not difficult to see that the series
$$\sum_{m=0}^{\infty}[1+k]k^{m}\rho(x_{0},x_{1})
+k^{m}\rho(x_{0},x_{2})$$ converges. Hence, $\rho(x_{n},x_{n+2p})$ converges as $n,p\ \ \text{go toward } \ \ \infty.$
Thus, by subcase 1 and subcase 2, we deduce that the sequence $\{x_{n}\}$ is a $\rho-$Cauchy sequence.
Since $(X,\rho)$ is a $\rho-$complete extended rectangular $b-$metric space, we deduce that $\{x_{n}\}$ converges to some $\nu \in X.$
We claim that $\nu$ is fixed point of $T.$
Note that, if there exists  an integer $N$ such that $x_{N}=\nu$. Due to case 2, $T^nx \neq \nu$ for all $n>N$. Similarly, $T^nx \neq T\nu$ for all $n>N$. Hence, we are in case 1, so $\nu$ is a fixed point of $T$.\\
Also, if there exists  an integer $N$ such that $T^Nx=T\nu$. Again, necessarily $T^nx\neq \nu$  and $T^nx\neq T\nu$ for all $n>N$. Thus, $T\nu=\nu$.
Therefore, we may assume that for all $n$ we have $x_{n}\not\in \{\nu, T\nu\}.$
\begin{align*}
\rho(\nu,T\nu)&\le \theta(\nu , T\nu , x_{n},x_{n+1})[\rho(\nu ,x_{n})+\rho(x_{n},x_{n+1})+ \rho(x_{n+1},T\nu)]\\ &\le
\theta(\nu , T\nu , x_{n},x_{n+1})[\rho(\nu ,x_{n})+\rho(x_{n},x_{n+1})+ \rho(Tx_{n},T\nu)]\\ &\le \theta(\nu , T\nu , x_{n},x_{n+1})[\rho(\nu ,x_{n})+\rho(x_{n},x_{n+1})+k\rho(x_{n},\nu)]
\end{align*}
Now, taking the limit as $n\rightarrow \infty$ we deduce that
$\rho(\nu,T\nu)=0$ and that is $T\nu =\nu$ and $\nu$ is a fixed point of $T$ as desired.

Finally, to show uniqueness assume there exist two fixed points of $T$ say $\nu $ and $\mu$ such that $\nu\neq \mu.$ By the contractive property of $T$ we have:
$$ \rho(\nu, \mu)=\rho(T\nu, T\mu)\le k \rho(\nu, \mu)<\rho(\nu, \mu)$$
which leads us to contradiction.Thus, $T$ has a unique fixed point as required.
\end{proof}
\begin{theorem}
Let $(X,\rho)$ be a complete extended rectangular $b-$metric space, and $T$ a self mapping on $X$ satisfying the following condition;
for all $x,y\in X$ there exists $0<k<\frac{1}{2}$  such that
$$ \rho(Tx,Ty)\le k[\rho(x,Tx)+\rho(y,Ty)]$$
Also, if
$$\sup_{m>1}\lim_{n\rightarrow \infty} \theta(x_{n},x_{n+1},x_{n+2},x_{m})\le \frac{1}{k},$$
and for all $u,v \in X$ we have:

$$ \lim_{n\rightarrow \infty} \theta(u,v,x_{n},x_{n+1})\le 1,$$
then $T$ has a unique fixed point in $X.$

\end{theorem}
\begin{proof}
Let $x_{0}\in X$ and define the sequence $\{x_{n}\}$ as follows
$$x_{1}=Tx_{0}, x_{2}=Tx_{1}=T^{2}x_{0}, \cdots, x_{n}=Tx_{n-1}=T^{n}x_{0}, \cdots $$
First of all, Note that for all $n\ge 1$ we have
\begin{align*}
&\rho(x_{n},x_{n+1})\le k[\rho(x_{n-1},x_{n})+\rho(x_{n},x_{n+1})]\\&\Rightarrow (1-k)\rho(x_{n},x_{n+1})\le k\rho(x_{n-1},x_{n})\\ &\Rightarrow \rho(x_{n},x_{n+1})\le \frac{k}{1-k}\rho(x_{n-1},x_{n}).
\end{align*}
Since $0<k<\frac{1}{2}$ one can easily deduce that $0<\frac{k}{1-k}<1.$ So, let $\mu =\frac{k}{1-k}$

Hence,
\begin{align*}
\rho(x_{n},x_{n+1})&\le \mu \rho(x_{n-1},x_{n})\\
&\le
\mu^{2}\rho(x_{n-2},x_{n-1})\\ & \le \cdots\\ & \le  \mu^{n}\rho(x_{0},x_{1}).
\end{align*}
Therefore,
$$\rho(x_{n},x_{n+1})\rightarrow 0\ \ \ \text{as} \ \ \ n\rightarrow \infty  \ \ \label{E3}$$

Also, for all $n\ge 1$ we have

$$\rho(x_{n},x_{n+2})\le k[\rho(x_{n-1},x_{n})+\rho(x_{n+1},x_{n+2})]$$
Thus, by using the fact that $\rho(x_{n},x_{n+1})\rightarrow 0\ \ \ \text{as} \ \ \ n\rightarrow \infty,$ we deduce that
$$\rho(x_{n},x_{n+2})\rightarrow 0\ \ \ \text{as} \ \ \ n\rightarrow \infty.$$

Now, similarly to the prove of case 1 and case 2 of
Theorem \ref{one}, we deduce that the sequence $\{x_{n}\}$ is a $\rho-$Cauchy sequence.
Since $(X,\rho)$ is a $\rho-$complete extended rectangular $b-$metric space, we conclude that $\{x_{n}\}$ converges to some $\nu \in X.$
Using the argument in the prove of Theorem \ref{one}, we may assume that for all $n\ge 1 \ \ \text{we have}\ \ x_{n}\not\in \{\nu,T\nu\}.$ Thus,
\begin{align*}
\rho(\nu,T\nu)&\le \theta(\nu , T\nu , x_{n},x_{n+1})[\rho(\nu ,x_{n})+\rho(x_{n},x_{n+1})+ \rho(x_{n+1},T\nu)]\\ &\le
\theta(\nu , T\nu , x_{n},x_{n+1})[\rho(\nu ,x_{n})+\rho(x_{n},x_{n+1})+ \rho(Tx_{n},T\nu)]\\ &\le \theta(\nu , T\nu , x_{n},x_{n+1})[\rho(\nu ,x_{n})+\rho(x_{n},x_{n+1})+k\rho(x_{n}, Tx_{n})+k\rho(\nu,T\nu)].
\end{align*}
Taking the limit of the above inequalities we get:

$$\rho(\nu,T\nu)\le [0+0+0+k\rho(\nu,T\nu)]< \rho(\nu, T\nu).$$
Thus, $\rho(\nu,T\nu)=0$ which implies that $T\nu=\nu$ and hence
$\nu$ is a fixed point of $T.$
Finally, to show uniqueness assume there exist two fixed points of $T$ say $\nu $ and $\mu$ such that $\nu\neq \mu.$ By the contractive property of $T$ we have:
$$ \rho(\nu, \mu)=\rho(T\nu, T\mu)\le k \rho(\nu, \mu)<\rho(\nu, \mu)$$
which leads us to contradiction.Thus, $T$ has a unique fixed point as required.
\end{proof}

\section{Application}
In closing, we present the following application for our results.
\begin{theorem}\label{tm}
For any natural number $m\ge 3$ the equation
\begin{equation}\label{app1}
x^{m}+1= (m^{4}-1)x^{m+1}+m^{4}x
\end{equation}
has a unique real solution.
\end{theorem}
\begin{proof}
First of all, note that if $|x|>1,$  Equation $(3.1),$ does not have a solution.
So, let $X=[-1,1]$ and for all $x,y\in X$ let $\rho(x,y)=|x-y|$\\
and $\alpha(x,y,u,v)=\max\{x,y,u,v\}+2.$ It is not difficult to see that
$(X,\rho)$ is a $\rho-$complete controlled rectangular $b-$metric space. Now, let
\begin{equation*}
Tx=\frac{x^{m}+1}{(m^{4}-1)x^{m}+m^{4}}
\end{equation*}
Notice that, since $m\ge 2,$ we can deduce that $m^{4}\ge 6.$ Thus,
\begin{align*}
\rho(Tx,Ty)&=|\frac{x^{m}+1}{(m^{4}-1)x^{m}+m^{4}}-\frac{y^{m}+1}{(m^{4}-1)y^{m}+m^{4}}| \\
& = |  \frac{x^{m}-y^{m}}{((m^{4}-1)x^{m}+m^{4})((m^{4}-1)y^{m}+m^{4})}|\\
&\le \frac{|x-y|}{m^{4}}\\
&\le \frac{|x-y|}{6}\\
&=  \frac{1}{6}\rho(x,y)
\end{align*}
Hence,
\begin{equation*}
\rho(Tx,Ty)\le  k \rho(x,y) \ \ \ \text{where} \ \ \ k=\frac{1}{6}
\end{equation*}
On the other hand, notice that for all $x_{0}\in X$ we have
\begin{equation*}
x_{n}=T^{n}x_{0}\le \frac{2}{m^{4}}
\end{equation*}
Thus,
\begin{eqnarray*}
\sup_{n\geq 1}\lim_{i\rightarrow \infty }\theta (x_{i},x_{i+1},x_{i+2},x_{n})
&= &\frac{2}{m^{4}}\\
&\le &2<6=\frac{1}{k}.
\end{eqnarray*}

Finally, note that $T$ satisfies all the hypothesis of Theorem \ref{one}. Therefore,
$T$ has a unique fixed point in $X,$ which implies that Equation $(3.1),$ has a unique real solution as desired.
\end{proof}

\section{Conclusion}
In closing, we would like to bring to the readers attention the following open questions;
\begin{question}
Let $(X,\rho)$ is a controlled rectangular $b-$metric space, and $T$ a self mapping on $X.$
Also, assume that for all distinct $x,y,Tx,Ty\in X$ there exists $k\in (0,1)$ such that
\begin{equation*}
\rho(Tx,Ty)\le k\theta(x,y,Tx,Ty)\rho(x,y)
\end{equation*}
what are the other hypothesis we should add so that $T$ has a unique fixed point in the whole space $X?$
\end{question}
\begin{question}
Let $(X,\rho)$ is a controlled rectangular $b-$metric space, and $T$ a self mapping on $X.$
Also, assume that for all distinct $x,y,Tx,Ty\in X$ there exists $k\in (0,1)$ such that
\begin{equation*}
\rho(Tx,Ty)\le \theta(x,y,Tx,Ty)[\rho(x,Tx)+\rho(y,Ty)]
\end{equation*}
what are the other hypothesis we should add so that $T$ has a unique fixed point in the whole space $X?$
\end{question}
\title{\textbf{Acknowledgements}}\\
The author would like to thank Prince Sultan University for funding this work through research group Nonlinear Analysis Methods in Applied Mathematics (NAMAM) group number RG-DES-2017-01-17.\\


\begin{thebibliography}{99}
\bibitem{G2}T. Abdeljawad, K. Abodayeh, N. Mlaiki, On fixed point generalizations to partial b-metric spaces,  Journal Comput. Anal. Appl. 19 (5) (2015),  883-891.
\bibitem{Kamran} T. Kamran,  M. Samreen,  Q. UL Ain, A Generalization of b-metric space and some fixed
point theorems, Mathematics, 5 (19) (2017), 1-7.
\bibitem{Branc} A. Branciari, A fixed point theorem of Banach-Caccioppoli type on a class of generalized metric spaces, Publicationes Mathematicae, (2000) (57) 1-2, 31-37.
\bibitem{Geo} R. George, s. Rodenovic, K. P. Rashma, S. Shukla,  Rectangular $b-$metric spaces and contraction principals, Journal of Nonlinear Science and Applications, (2015) (8) 1005-1013.
\end{thebibliography}
\end{document}